\documentclass[a4paper]{article}
\usepackage{amsthm,amsmath,amssymb}
\usepackage{enumerate,enumitem}
\usepackage{graphicx}

\usepackage[all]{xy}
\usepackage{xfrac}
\usepackage{ulem}
\usepackage{mathtools}
\usepackage{ wasysym }
\usepackage{tabularx,booktabs,multirow}
\usepackage{float}
\usepackage{hyperref}
\newcolumntype{C}{>{\centering\arraybackslash}X}
\newcolumntype{D}{>{\centering\arraybackslash}X}

\DeclareGraphicsExtensions{.pdf,.png,.eps}
\graphicspath{{figs/}}

\newtheorem{theorem}{Theorem}
\newtheorem{lemma}[theorem]{Lemma}

\newtheorem{corollary}[theorem]{Corollary}

\newtheorem*{claim*}{Claim}

\theoremstyle{remark}

\newcommand{\bw}{\ensuremath{\!\!\between\!\!}}

\newcommand{\N}{\ensuremath{\mathbb{N}}}

\newcommand{\R}{\ensuremath{\mathbb{R}}}

\newcommand{\cX}{\ensuremath{\mathcal{X}}}
\newcommand{\cY}{\ensuremath{\mathcal{Y}}}

\newcommand{\PP}{\ensuremath{\mathcal{P}}}
\newcommand{\QQ}{\ensuremath{\mathcal{Q}}}

\newcommand{\CC}{\ensuremath{\mathcal{C}}}

\usepackage[usenames,dvipsnames]{xcolor}

\begin{document}

\title{Poset Ramsey number $R(P,Q_n)$. I.\linebreak Complete multipartite posets}

\author{Christian Winter\thanks{Karlsruhe Institute of Technology, Karlsruhe, Germany, E-mail:\ \textit{christian.winter@kit.edu}. Research was partially supported by DFG grant FKZ AX 93/2-1.}}

\maketitle

\begin{abstract}
A poset $(P',\le_{P'})$ contains a copy of some other poset $(P,\le_P)$ if there is an injection $f\colon P'\to P$ where for every $X,Y\in P$, $X\le_P Y$ if and only if $f(X)\le_{P'} f(Y)$.
For any posets $P$ and $Q$, the poset Ramsey number $R(P,Q)$ is the smallest integer $N$ such that any blue/red coloring of a Boolean lattice of dimension $N$
contains either a copy of $P$ with all elements blue or a copy of $Q$ with all elements red.
We denote by $K_{t_1,\dots,t_\ell}$ a complete $\ell$-partite poset, i.e.\ a poset consisting of $\ell$ pairwise disjoint sets $A^i$ of size $t_i$, $1\le i\le \ell$, 
such that for any $i,j\in\{1,\dots,\ell\}$ and any two $X\in A^{i}$ and $Y\in A^{j}$, $X<Y$ if and only if $i<j$.
In this paper we show that $R(K_{t_1,\dots,t_\ell},Q_n)\le n+\frac{(2+o_n(1))\ell n}{\log n}$. 
\end{abstract}

\section{Introduction}

Ramsey theory is a field of combinatorics that asks whether in any coloring of the elements in a discrete host structure we find a particular monochromatic substructure.
This question offers a lot of variations depending on the chosen sub- and host structure.
While originating from a result of Ramsey \cite{R} on uniform hypergraphs from 1930, the most well-known setting considers monochromatic subgraphs in edge-colorings of complete graphs.
In contrast, this paper considers a Ramsey-type problem using partially ordered sets, or \textit{posets} for short, as the host structure.
A \textit{poset} is a set $P$ which is equipped with a relation $\le_P$ on the elements of $P$ that is transitive, reflexive, and antisymmetric. 
Whenever it is clear from the context we refer to such a poset $(P,\le_P)$ just as $P$.
Given a non-empty set $\cX$, the poset consisting of all subsets of $\cX$ equipped with the inclusion relation $\subseteq$ is the \textit{Boolean lattice} $\QQ(\cX)$
of \textit{dimension} $|\cX|$. We use $Q_n$ to denote a Boolean lattice with an arbitrary $n$-element ground set. 
\\

We say that a poset $P_1$ is an \textit{induced subposet} of another poset $P_2$ if $P_1\subseteq P_2$ and for every two $X,Y\in P_1$, 
$$X \leq_{P_1}  Y\text{ if and only if }X \leq_{P_2} Y.$$
A \textit{copy} of $P_1$ in $P_2$ is an induced subposet $P'$ of $P_2$ which is isomorphic to $P_1$.\\
Here we consider color assignments of the elements of a poset $P$ using the colors \textit{blue} and \textit{red}, 
i.e.\ mappings $c\colon P \rightarrow \{\text{blue}, \text{red}\}$, which we refer to as a \textit{blue/red coloring} of $P$.
A poset is colored \textit{monochromatically} if all its elements have the same color. 
If a poset is colored monochromatically in blue [red], we say that it is a \textit{blue} [\textit{red}] \textit{poset}.
The elements of a poset $P$ are usually referred to as \textit{vertices}.
\\

Axenovich and Walzer \cite{AW} were the first to consider the following Ramsey variant on posets.
For posets $P$ and $Q$, the \textit{poset Ramsey number} of $P$ versus $Q$ is given by
\begin{multline*}
R(P,Q)=\min\{N\in\N \colon \text{ every blue/red coloring of $Q_N$ contains either}\\ 
\text{a blue copy of $P$ or a red copy of $Q$}\}.
\end{multline*}

As a central focus of research in this area, bounds on the poset Ramsey number $R(Q_n,Q_n)$ were considered and gradually improved 
with the best currently known bounds being\linebreak $2n+1 \leq R(Q_n, Q_n) \leq n^2 -n+2$, 
see listed chronologically Walzer \cite{W}, Axenovich and Walzer \cite{AW}, Cox and Stolee \cite{CS}, Lu and Thompson \cite{LT}, Bohman and Peng \cite{BP}. 
The related off-diagonal setting $R(Q_m,Q_n)$, $m<n$, also received considerable attention over the last years.
When both $m$ and $n$ are large, the best known upper bound is due to Lu and Thompson \cite{LT}, yielding together with a trivial lower bound that
 $m+n\le R(Q_m,Q_n)\le \big(m-2+o(1)\big)n+m$.
When $m$ is fixed and $n$ is large, an exact result is only known in the trivial case $m=1$ where $R(Q_1, Q_n)=n+1$.
For $m=2$, after earlier estimates by Axenovich and Walzer \cite{AW} as well as Lu and Thompson \cite{LT},
 the best known upper bound is due to Gr\'osz, Methuku, and Tompkins \cite{GMT},
which is complemented by a lower bound shown recently by Axenovich and the present author \cite{QnV}:
$$n \left(1 + \frac{1}{15 \log n}\right) \le R(Q_2,Q_n) \le n\left(1 + \frac{2+o(1)}{\log n}\right).$$

%
%
%

In this paper we generalize the upper bound of Grósz, Methuku and Tompkins \cite{GMT} on $R(Q_2,Q_n)$ to a broader class of posets, namely 
we discuss the poset Ramsey number of a \textit{complete multipartite poset} versus the Boolean lattice $Q_n$.
A \textit{complete $\ell$-partite poset} $K_{t_1,\dots,t_\ell}$ is a poset on $\sum_{i=1}^\ell t_i$ vertices obtained as follows. 
Consider $\ell$ pairwise disjoint \textit{layers} $A^1,\dots,A^\ell$ of vertices, where layer $A^i$ consists of $t_i$ distinct vertices.
Now for any two indexes $i,j\in\{1,\dots,\ell\}$ and any vertices $X\in A^i$, $Y\in A^j$, let $X<Y$ if and only if $i<j$.
Such a poset can be seen as a complete blow-up of a chain.
Note that $Q_2=K_{1,2,1}$.

\begin{figure}[H]
\centering
\includegraphics[scale=0.6]{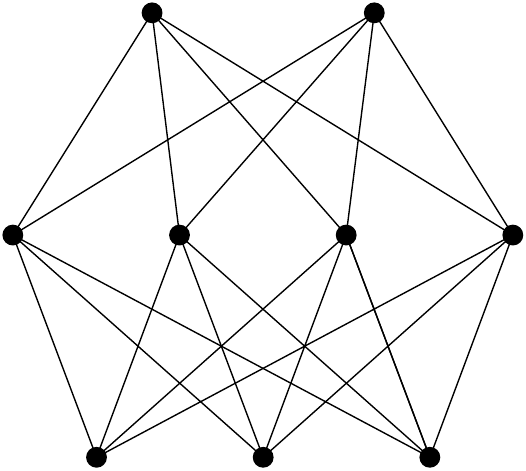}
\caption{Hasse diagram of the complete $3$-partite poset $K_{3,4,2}$}
\end{figure}

\begin{theorem}\label{thm_multipartite}
For $n\in\N$, let $\ell\in\N$ be an integer such that $\ell=o(\log n)$ and
 for $i\in\{1,\dots,\ell\}$, let $t_i\in\N$ be integers with $\sup_i t_i =n^{o(1)}$. Then
$$R(K_{t_1,\dots,t_\ell},Q_n)\le n\left(1+\frac{2+o(1)}{\log n}\right)^\ell\le n+\frac{\big(2+o(1)\big)\ell n}{\log n}.$$
\end{theorem}


Here and throughout this paper, the $O$-notation is used exclusively depending on $n$, i.e.\ $f(n)=o(g(n))$ if and only if $\frac{f(n)}{g(n)}\to 0$ for $n\to\infty$.
For parameters as above, this theorem implies that $R(K_{t_1,\dots,t_\ell},Q_n)=n+o(n)$.
Under the precondition that $\ell$ is fixed, we even obtain a bound that is asymptotically tight in the first and second summand:
We say that a complete $\ell$-partite poset $K=K_{t_1,\dots,t_\ell}$ is \textit{non-trivial}, if it is neither a chain nor an antichain, i.e.\ if $\ell\ge 2$ and $t_i\ge 2$ for some $i\in\{1,\dots,\ell\}$.
Observe that such a non-trivial $K$ contains either a copy of $K_{1,2}$ or $K_{2,1}$, so Theorem 2 of \cite{QnV} yields $R(K,Q_n)\ge n + \frac{n}{15 \log n}$.
Thus for non-trivial~$K$, $R(K,Q_n)=n+\Theta\left(\frac{n}{\log n}\right)$.
For trivial $K$, it is known that $R(K,Q_n)=n+\Theta(1)$. In detail, if $K$ is a chain on $\ell$ vertices, then $R(K,Q_n)= n+\ell-1$, 
where the upper bound is a consequence of Lemma \ref{chain_lem} stated later on and the lower bound is easy to see using a layered coloring of the host lattice.
If $K$ is an antichain on $t$ vertices, then a trivial lower bound, Lemma 3 in Axenovich and Walzer's \cite{AW}, and Sperner's Theorem imply $n\le R(K,Q_n)\le n+\alpha(t)$
where $\alpha(t)$ is the smallest integer such that $\binom{\alpha(t)}{\lfloor\alpha(t)/2\rfloor}\ge t$. 
\\

We shall first consider a special complete multipartite poset that we call a \textit{spindle}.
Given $r\ge 0$, $s\ge1$ and $t\ge 0$, an \textit{$(r,s,t)$-spindle} $S_{r,s,t}$ is defined as the complete multipartite poset $K_{t'_1,\dots,t'_{r+1+t}}$ where
 $t'_1,\dots,t'_r=1$ and $t'_{r+1}=s$ and $t'_{r+2},\dots,t'_{r+1+t}=1$.
 In other words this poset on $r+s+t$ vertices is constructed using an antichain $A$ of size $s$ and two chains $C_r,C_t$ on $r$ and $t$ vertices, respectively, 
 combined such that every vertex of $A$ is larger than every vertex from $C_r$ but smaller than every vertex from $C_t$.
 
\begin{figure}[H]
\centering
\includegraphics[scale=0.6]{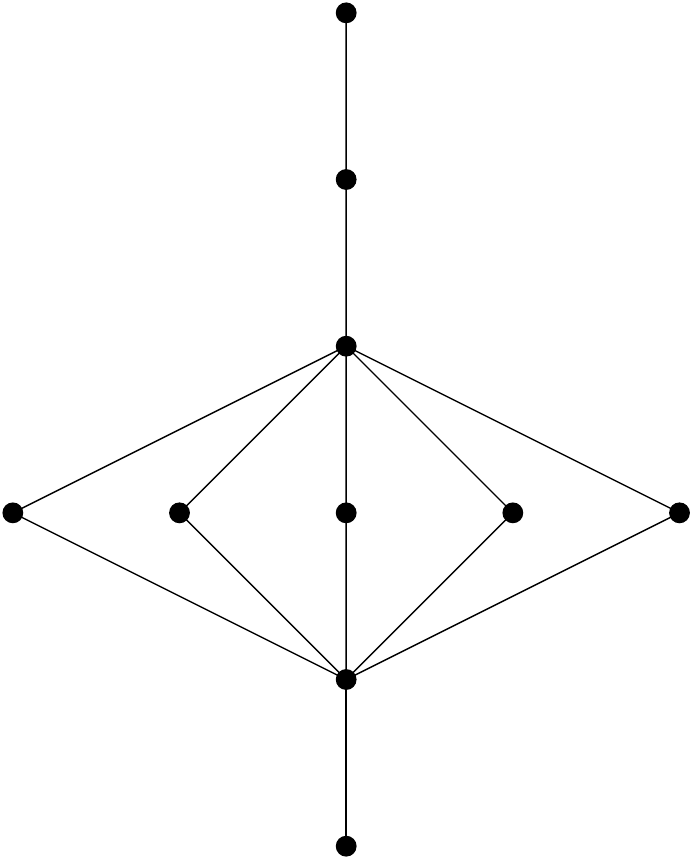}
\caption{Hasse diagram of the spindle $S_{2,5,3}$}
\end{figure}

 \begin{theorem}\label{thm_spindle}
Let $r,s,t$ be non-negative integers with $r+t=o(\sqrt{\log n})$ and $1\le s=n^{o(1)}$ for $n\in\N$. Then 
$$R(S_{r,s,t},Q_n)\le n+\frac{\big(1+o(1)\big)(r+t)n}{\log n}.$$
\end{theorem}



The spindle $S_{1,s,1}$ is known in the literature as an \textit{$s$-diamond} $D_s$, while the poset $S_{1,s,0}$ is usually referred to as an \textit{$s$-fork} $V_s$.

\begin{corollary}\label{cor_fork}\label{cor_diamond}
Let $s\in\N$ with $s=n^{o(1)}$ for $n\in\N$. Then
$$R(D_s,Q_n)\le n+\frac{\big(2+o(1)\big)n}{\log n}\qquad\text{ and }\qquad R(V_s,Q_n)\le n+\frac{\big(1+o(1)\big)n}{\log n}.$$
\end{corollary}

%
%

For a positive integer $n\in\N$, we use $[n]$ to denote the set $\{1,\dots,n\}$, additionally let $[0]=\varnothing$.
Here `$\log$' always refers to the logarithm with base $2$. 
We omit floors and ceilings where appropriate.

The structure of the paper is as follows. First, in Section \ref{sec_prelim} we introduce some notation and two preliminary lemmas. 
In Section \ref{sec_multipartite} we show the bound for spindles and afterwards the generalization for general complete multipartite posets.


\section{Preliminaries}\label{sec_prelim}


\subsection{Red $Q_n$ versus blue chain}\label{sec_chain}

Let $\cX$ and $\cY$ be disjoint sets. 
Then the vertices of the Boolean lattice $\QQ(\cX\cup\cY)$, i.e.\ the subsets of $\cX\cup\cY$, can be partitioned with respect to $\cX$ and $\cY$ in the following manner.
Every $Z\subseteq\cX\cup\cY$ has an $\cX$\textit{-part} $X_Z=Z\cap \cX$ and a $\cY$\textit{-part} $Y_Z=Z\cap\cY$.
In this setting, we refer to $Z$ alternatively as the pair $(X_Z,Y_Z)$. 
Conversely, for all $X\subseteq\cX$, $Y\subseteq\cY$, the pair $(X,Y)$ corresponds uniquely to the vertex $X\cup Y\in \QQ(\cX\cup\cY)$.
One can think of such pairs as elements of the Cartesian product $2^\cX \times 2^\cY$ which has a canonical bijection to $2^{\cX\cup\cY}=\QQ(\cX\cup\cY)$.
Observe that for $X_i\subseteq\cX, Y_i\subseteq\cY$, $i\in[2]$, we have $(X_1,Y_1)\subseteq (X_2,Y_2)$ if and only if $X_1\subseteq X_2$ and $Y_1\subseteq Y_2$. 
\\

\noindent We shall need the following lemma.

\begin{lemma}\label{chain_lem}
Let $\cX$, $\cY$ be disjoint sets with $|\cX|=n$ and $|\cY|=k$, for some $n,k\in\N$. Let $\QQ=\QQ(\cX\cup\cY)$ be a blue/red colored Boolean lattice.
Fix some linear ordering  $\pi=(y_1,\dots,y_k)$  of $\cY$ and define $Y(0), \ldots, Y(k)$ by $Y(0)=\varnothing$ and  $Y(i)=\{y_1,\dots,y_i\}$ for $i\in[k]$.
Then there exists at least one of the following in $\QQ$:
\renewcommand{\labelenumi}{(\alph{enumi})}
\begin{enumerate}
\item a red copy of $Q_n$, or
\item a blue chain of length $k+1$ of the form $(X_0,Y(0)),\dots,(X_{k},Y(k))$ where $X_0 \subseteq X_1 \subseteq\dots \subseteq X_k\subseteq \cX$.
\end{enumerate}
\end{lemma}

Note that a version of this lemma was used implicitly in a paper of Gr\'osz, Methuku and Tompkins \cite{GMT}.
It was stated explicitly and reproved by Axenovich and the author, see Lemma 8 in \cite{QnV}.

\subsection{Gluing two posets}
By identifying vertices of two posets, they can be ``glued together'' creating a new poset. 
We will later construct complete multipartite posets by gluing spindles on top of each other using the following definition. 
Given a poset $P_1$ with a unique maximal vertex $Z_1$ and a poset $P_2$ disjoint from $P_1$ with a unique minimal vertex $Z_2$, 
let $P_1\bw P_2$ be the poset obtained by identifying $Z_1$ and $Z_2$.
Formally speaking, $P_1\bw P_2$ is the poset $(P_1\setminus\{Z_1\}) \cup (P_2\setminus \{Z_2\}) \cup \{Z\}$ for a $Z\notin P_1\cup P_2$
where for any two $X,Y\in P_1\bw P_2$, $X<_{P_1\!\between\! P_2}Y$ if and only if one of the following five cases hold:
$X,Y\in P_1 $ and $X<_{P_1}Y$; $X,Y\in P_2$ and $X<_{P_2} Y$; $X\in P_1$ and $Y \in P_2$; $X\in P_1$ and $Y=Z$; or $X=Z$ and $Y\in P_2$.

\begin{figure}[H]
\centering
\includegraphics[scale=0.55]{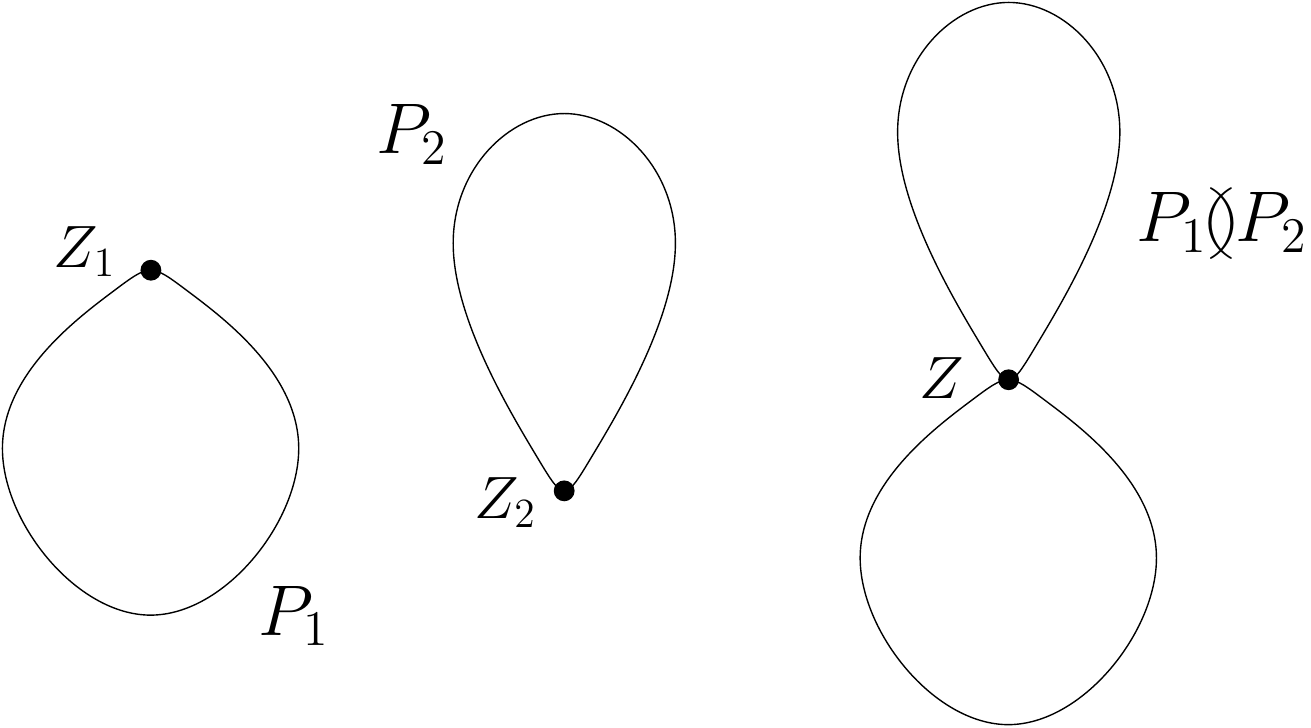}
\caption{Creating $P_1\bw P_2$ from $P_1$ and $P_2$}
\end{figure}

\begin{lemma}\label{lem_gluing}
Let $P_1$ be a poset with a unique maximal vertex and let $P_2$ be a poset with a unique minimal vertex. 
Then $R(P_1\bw P_2,Q_n)\le R(P_1,Q_{R(P_2,Q_n)})$.
\end{lemma}

\begin{proof}
Let $N=R(P_1,Q_{R(P_2,Q_n)})$. Consider a blue/red colored Boolean lattice $\QQ$ of dimension $N$ which contains no blue copy of $P_1\bw P_2$. 
We shall prove that there exists a red copy of $Q_n$ in this coloring.
We say that a blue vertex $X$ in $\QQ$ is \textit{$P_1$-clear} if there is no blue copy of $P_1$ in $\QQ$ containing $X$ as its maximal vertex.
Similarly, a blue vertex $X$ is \textit{$P_2$-clear} if there is no blue copy of $P_2$ in $\QQ$ with minimal vertex $X$.
Observe that every blue vertex is $P_1$-clear or $P_2$-clear (or both), since there is no blue copy of $P_1\bw P_2$.
\\

We introduce an auxiliary coloring of $\QQ$ using colors green and yellow. 
Color all blue vertices which are $P_1$-clear in green and all other vertices in yellow. 
Then this coloring does not contain a monochromatic green copy of $P_1$, since otherwise the maximal vertex of such a copy is not $P_1$-clear.
Recall that $N=R(P_1,Q_{R(P_2,Q_n)})$, thus $\QQ$ contains a monochromatic yellow copy of $Q_{R(P_2,Q_n)}$, which we refer to as $\QQ'$.
\\

Consider the original blue/red coloring of $\QQ'$. Every blue vertex of $\QQ'$ is yellow in the auxiliary coloring, i.e.\ not $P_1$-clear. 
Thus every blue vertex of $\QQ'$ is $P_2$-clear. 
This coloring of $\QQ'$ does not contain a blue copy of $P_2$, since otherwise the minimal vertex of such a copy is not $P_2$-clear.
Note that the Boolean lattice $\QQ'$ has dimension $R(P_2,Q_n)$, thus there exists a monochromatic red copy of $Q_n$ in $\QQ'$, hence also in $\QQ$.
\end{proof}

\begin{corollary}\label{cor_gluing}
Let $P_1$ be a poset with a unique maximal vertex and let $P_2$ be a poset with a unique minimal vertex. 
Suppose that there are functions $f_1,f_2\colon \N\to \R$ with $R(P_1,Q_n)\le f_1(n)n$ and $R(P_2,Q_n)\le f_2(n)n$ for any $n\in\N$ and
such that $f_1$ is monotonically non-increasing.
Then for every $n\in\N$,
$$R(P_1\bw P_2,Q_n)\le f_1(n)f_2(n)n.$$
\end{corollary}
\begin{proof}
For an arbitrary $n\in\N$, let $n'= f_2(n)n$. 
Note that for any poset $P$, $R(P,Q_n)\ge n$, so $n'\ge n$.
Hence $f_1(n')\le f_1(n)$, and Lemma \ref{lem_gluing} provides
$$R(P_1\bw P_2,Q_n)\le R(P_1,Q_{n'})\le f_1(n')n'\le f_1(n)f_2(n)n.$$
\end{proof}

\section{Proofs of Theorem \ref{thm_spindle} and Theorem \ref{thm_multipartite}} \label{sec_multipartite}


\begin{proof}[Proof of Theorem \ref{thm_spindle}]
Let $\epsilon=\frac{\log s}{\log n}$, so $s=n^{\epsilon}$ and $\epsilon=o(1)$. We can suppose that $n$ is large\linebreak and hence $\epsilon<1$.
Then let $c=\frac{r+t+\delta}{1-\epsilon}$ where $\delta=\frac{2(r+1)}{\log n}(\log\log n +r+t)$. Since $r+t=o(\sqrt{\log n})$, $\delta=o(1)$.
Let $k=\frac{cn}{\log n}$. We show for sufficently large $n$ that $R(S_{r,s,t},Q_n)\le n+k$.
If $s=1$, $S_{r,s,t}$ is a chain and $R(S_{r,s,t},Q_n)\le n+r+s\le n+k$ by Lemma \ref{chain_lem}, so suppose $s\ge2$.\\

\noindent \textit{Claim:} For sufficiently large $n$, $k!>2^{(r+t)(n+k)}\cdot (s-1)^{k+1}$.\\
Note that $k!>\left(\frac{k}{e}\right)^k=2^{k(\log k -\log e)}$ and $(s-1)^{k+1}=2^{(k+1)\log (s-1)}$.
Thus we shall prove that $k(\log k -\log e)>(r+t+\log(s-1))k+\log (s-1)+(r+t)n$.
Using that $k=\frac{cn}{\log n}$ and $s-1\le n^{\epsilon}$, we obtain
\begin{align*}
&k\big(\log k-\log (s-1)\big)-k\big(r+t+\log e\big)-\log (s-1)-\big(r+t\big)n\\
&\ge \frac{cn}{\log n} \big(\log c + \log n -\log\log n -\epsilon \log n\big)-\frac{cn }{\log n}\big(r+t+\log e\big)-\epsilon \log n -\big(r+t\big)n\\
&\ge cn\big(1-\epsilon \big)-\big(r+t\big)n-\frac{cn }{\log n}\big(\log\log n +r+t+\log e\big)-\epsilon \log n \\
&> \delta n- \frac{2(r+1)n}{\log n}\big(\log\log n +r+t\big)=0,
\end{align*}
where the last inequality holds for sufficiently large $n$.\qed
\\

Let $\cX$ and $\cY$ be disjoint sets with $|\cX|=n$ and $|\cY|=k$. We consider a blue/red coloring of $\QQ=\QQ(\cX\cup\cY)$ with no red copy of $Q_n$. 
We shall show that there is a monochromatic blue copy of $S_{r,s,t}$ in $\QQ$.
For every linear ordering $\pi=(y_1^\pi,\dots,y_k^\pi)$ of $\cY$, Lemma \ref{chain_lem} provides a blue chain $C^\pi$ 
of the form $Z^\pi_0=(X^\pi_0,\varnothing), Z^\pi_1=(X^\pi_1,\{y_1^{\pi}\}),\dots,Z^\pi_k=(X^\pi_k,\cY)$, where $X^\pi_i\subseteq \cX$.
\\

For every ordering $\pi$ of $\cY$, we consider the $r$ smallest vertices $Z^{\pi}_0,\dots, Z^{\pi}_{r-1}$ 
and the $t$ largest vertices $Z^{\pi}_{k-t+1},\dots, Z^{\pi}_{k}$ of its corresponding chain $C^{\pi}$, so let $I=\{0,\dots,r-1\}\cup\{k-t+1,\dots,k\}$.
Each $Z^{\pi}_{i}$ is a vertex of $\QQ$, so one of the $2^{n+k}$ distinct subsets of $\cX\cup\cY$.
Thus for a fixed $\pi$, there are at most $\left(2^{n+k}\right)^{r+t}$ distinct combinations of the $Z^{\pi}_{i}$, $i\in I$.
Recall that $k!>2^{(r+t)(n+k)}\cdot (s-1)^{k+1}$. By pigeonhole principle, we find a collection $\pi_1,\dots,\pi_m$ of $m= (s-1)^{k+1}+1$ distinct linear orderings of $\cY$
such that for all $j\in[m]$ and $i\in I$, $Z^{\pi_{j}}_{i}=Z_i$ for some $Z_i\subseteq\cX\cup\cY$ independent of $j$. 
In other words, we find many chains with same $r$ smallest vertices $Z_i$, $i\in \{0,\dots,r-1\}$, and same $t$ largest vertices $Z_i$, $i\in \{k-t+1,\dots,k\}$.
Let $\PP$ be the poset induced in $\QQ$ by the chains $C^{\pi_j}$, $j\in[m]$.
\\

If there is an antichain $A$ of size $s$ in $\PP$, then none of the vertices $Z_i$, $i\in I$, is in $A$, because they are contained in every chain $C^{\pi_j}$ and therefore comparable to all other vertices in $\PP$. Now $A$ together with the vertices $Z_i$, $i\in I$, form a copy of $S_{r,s,t}$ in $\PP$. Recall that all vertices in every $C^{\pi_j}$ are blue, i.e.\ $\PP$ is monochromatic blue. 
Thus we obtain a blue copy of $S_{r,s,t}$ in $\QQ$, so we are done.
From now on, suppose that there is no antichain of size $s$ in~$\PP$.
By Dilworth's Theorem we obtain $s-1$ chains $\CC_1,\dots,\CC_{s-1}$ which cover all vertices of $\PP$, i.e.\ all vertices of the $C^{\pi_j}$'s.
Note that the chains $\CC_i$ might consist of significantly more vertices than the $(k+1)$-element chains $C^{\pi_j}$.
\\

Now we consider the restriction to $\cY$ of each vertex in $\PP$, i.e.\ the sets $Z^{\pi}_{i}\cap\cY$, in order to apply the pigeonhole principle once again.
Assume for a contradiction that for some $i\in[s-1]$ there are $Z,Z'\in\CC_i$ with $|Z\cap\cY|=|Z'\cap\cY|$ but $Z\cap\cY\neq Z'\cap\cY$.
This implies that $Z\cap\cY\nsubseteq Z'\cap\cY$ and $Z\cap\cY\nsupseteq Z'\cap\cY$, so $Z$ and $Z'$ are incomparable, a contradiction as they are both contained in the chain $\CC_i$.
Consequently, there is only at most one $\ell$-element set $Y_i^\ell\subseteq \cY$, $\ell\in\{0,\dots,k\}$, for which there exists a $Z\in\CC_i$ with $Z\cap\cY=Y_i^\ell$.
\\

Note that for all $j\in[m]$ and for all $\ell\in\{0,\dots,k\}$, $|Z^{\pi_j}_\ell\cap \cY|=\ell$, i.e.\ $Z^{\pi_j}_\ell\cap \cY=Y_i^\ell$ for some $i\in[s-1]$.
In other words, for fixed $j$, each of the $k+1$ sets $Z^{\pi_j}_\ell\cap \cY$, $\ell\in\{0,\dots,k\}$, is equal to one of at most $s-1$ $Y_i^\ell$'s.
Recall that we have chosen $m=(s-1)^{k+1}+1$ distinct linear orderings $\pi_j$ of $\cY$.
Using pigeonhole principle we find two indexes $j_1,j_2$ such that $Z^{\pi_{j_1}}_\ell\cap \cY=Z^{\pi_{j_2}}_\ell\cap \cY$ for all $\ell\in\{0,\dots,k\}$.
This implies that $y^{\pi_{j_1}}_\ell=y^{\pi_{j_2}}_\ell$, i.e.\ $\pi_{j_1}$ and $\pi_{j_2}$ are equal.
But this is a contradiction to the fact that all orderings $\pi_j$ are distinct.

%
\end{proof}



Now we extend Theorem \ref{thm_spindle} to general complete multipartite posets using Corollary \ref{cor_gluing}.

\begin{proof}[Proof of Theorem \ref{thm_multipartite}]
Let $t=\sup_i t_i$. Then Theorem \ref{thm_spindle} shows the existence of a function $\epsilon(n)=o(1)$ with $R(K_{1,t,1},Q_n)\le n\left(1 + \frac{2+\epsilon(n)}{\log n}\right)$. 
We can suppose that $\epsilon$ is monotonically non-increasing by replacing $\epsilon(n)$ with $\max_{N>n} \{\epsilon(N),0\}$ where necessary.
In order to prove the theorem, we show a stronger statement using the auxiliary $(2\ell+1)$-partite poset $P=K_{1,t,1,t,\dots,1,t,1}$.
Note that $K_{t_1,\dots,t_\ell}$ is an induced subposet of $P$, thus $R(K_{t_1,\dots,t_\ell},Q_n)\le R(P,Q_n)$.
In the following we verify that $$R(P,Q_n)\le n\left(1+\frac{2+\epsilon(n)}{\log n}\right)^\ell.$$

We use induction on $\ell$. If $\ell=1$, then $P=K_{1,t,1}$, so $R(P,Q_n)\le n\left(1 + \frac{2+\epsilon(n)}{\log n}\right)$. 
If $\ell\ge 2$, we ``deconstruct'' the poset into two parts. Consider $P_1=K_{1,t,1}$ and the complete $(2\ell-1)$-partite poset $P_2=K_{1,t,1,t,\dots,1,t,1}$.
Then $P_1$ has a unique maximal vertex and $P_2$ has a unique minimal vertex. Observe that $P_1\bw P_2=P$.
Using the induction hypothesis
$$R(P_1,Q_n)\le n\left(1 + \frac{2+\epsilon(n)}{\log n}\right)\text{ and }R(P_2,Q_n)\le n\left(1+\frac{2+\epsilon(n)}{\log n}\right)^{\ell-1}.$$
Now Corollary \ref{cor_gluing} provides the required bound.
\end{proof}

\section{Conclusive remarks}

In this paper we considered $R(K,Q_n)$ where $K$ is a complete multipartite poset.
Although the presented bounds hold if the parameters of $K$ depend on $n$,
the original motivation for these results concerned the case where $K$ is fixed, i.e.\ independent from $n$:

After $R(Q_2,Q_n)$ was bounded asymptotically sharply by Gr\'osz, Methuku and Tompkins \cite{GMT} and Axenovich and the present author \cite{QnV},
the examination of $R(Q_3,Q_n)$ is an obvious follow-up question.
The best known upper bound is due to Lu and Thompson \cite{LT}, while the best known lower bound can be deduced from a bound on $R(K_{1,2},Q_n)$ in \cite{QnV},
$$n+\tfrac{n}{15\log n}\le R(K_{1,2},Q_n)\le R(Q_3,Q_n)\le \tfrac{37}{16}n+\tfrac{39}{16}.$$
In order to find better bounds and answer the question whether or not $R(Q_3,Q_n)=n+o(n)$, the consideration of $R(P,Q_n)$ for small posets $P$ might prove helpful as building blocks for Boolean lattices. For example, $Q_3$ can be partitioned into a copy of $K_{1,3}$ and a copy of $K_{3,1}$ which interact in a proper way. 
Both of these posets are complete $2$-partite posets with, as shown here, Ramsey numbers bounded by $$R(K_{1,3},Q_n)=R(K_{3,1},Q_n)=n+\Theta\left(\frac{n}{\log n}\right).$$
However, it remains open how to use our estimate to tighten the bounds on $R(Q_3,Q_n)$.
\\


\noindent \textbf{Acknowledgments:}~\quad  The author would like to thank Maria Axenovich for helpful\linebreak discussions and comments on the manuscript. 



\end{document}